\documentclass[12pt,reqno]{amsart}

\usepackage{epsfig}
\usepackage{amsmath, amsthm, amsfonts, amssymb}
\usepackage{graphicx}
\usepackage{epstopdf}

\numberwithin{equation}{section}

\newcommand{\C}{{\mathbb C}}

\newcommand{\Z}{{\mathbb Z}}

\newcommand{\dcal}{{\mathcal D}}

\newtheorem{theo}{{\sc \bf Theorem}}[section]

\newtheorem{lem}[theo]{{\sc \bf Lemma}}
\newtheorem{prop}[theo]{{\sc \bf Proposition}}

\begin{document}

\title{Implementations of Derivations on the Quantum Annulus}

\author[Klimek]{Slawomir Klimek}
\address{Department of Mathematical Sciences,
Indiana University-Purdue University Indianapolis,
402 N. Blackford St., Indianapolis, IN 46202, U.S.A.}
\email{sklimek@math.iupui.edu}

\author[McBride]{Matt McBride}
\address{Department of Mathematics and Statistics,
Mississippi State University,
175 President's Cir., Mississippi State, MS 39762, U.S.A.}
\email{mmcbride@math.msstate.edu}

\author[Sakai]{Kaoru Sakai}
\address{Department of Mathematical Sciences,
Indiana University-Purdue University Indianapolis,
402 N. Blackford St., Indianapolis, IN 46202, U.S.A.}
\email{ksakai@iupui.edu}

\date{\today}

\begin{abstract}
We construct compact parametrix implementations of covariant derivations on the quantum annulus. 
\end{abstract}

\maketitle

\section{Introduction}
The goal of this paper is to provide simple examples of Dirac type operators on noncommutative compact manifolds.
We study analogs of d-bar operators on the quantum annulus using only inherent geometrical structures: rotations, invariant states, covariant derivations and their implementations. In our previous paper \cite{KM}, we constructed similar d-bar type operators on the quantum annulus using APS-type boundary conditions. The class of operators was designed to mimic the classical Atiyah-Patodi-Singer theory and is different, less geometrical than the class studied in this paper. The main outcome, as in the past paper, is that we show that our quantum d-bar type operators have compact parametrices, like elliptic differential operators on compact manifolds.

Our another paper \cite{KMR} on quantum annulus, following \cite{KMRSW}, contains a description of unbounded derivations, covariant with respect to a natural rotation, and their implementations in Hilbert spaces obtained from the GNS construction with respect to invariant states. It turned out that no such implementation in any GNS Hilbert space for a faithful, normal, invariant state has compact parametrices for a large class of boundary conditions. However, as demonstrated in \cite{KMP}, if we relax the concept of an implementation by allowing operators to act between different Hilbert spaces, then there is an interesting class of examples of quantum d-bar operators with compact parametrices that can be constructed this way for the case of the quantum disk.  It is the purpose of this paper to extend those ideas to the quantum annulus.

Spectral triples are a key tool in noncommutative geometry \cite{Connes}, as they allow using analytical methods in studying quantum spaces. Since compact parametrix property is a part of the spectral triples conditions, our papers \cite{KMRSW} and \cite{KMR} demonstrate, using analytic techniques, that spectral triples, in general, cannot be constructed on the quantum disk and the quantum annulus using implementations of covariant derivations in GNS Hilbert spaces. Another, topological reason was pointed out in \cite{KMP2}, in the case of the quantum disk $\mathcal{T}$. Namely, the pull-back map in K-Homology $K^0(C(S^1))\to K^0(\mathcal{T})$ is an isomorphism and so the restriction map $K^0(\mathcal{T})\to K^0(\mathcal{K})$ is a zero map. Consequently, any spectral triple over the Toeplitz algebra, when restricted to the ideal of compact operators $\mathcal{K}$ should be trivial in K-Homology. However, it is easy to compute that implementations of covariant derivations pair nontrivially with a minimal projection in $\mathcal{K}$, and hence they cannot lead to spectral triples over $\mathcal{T}$. Similar arguments seem to also apply to the quantum annulus.

Additionally, as pointed out in \cite{FMR}, there are fundamental reasons why APS boundary conditions are not compatible with spectral triples even in classical geometry for algebras of functions which are non-constant on the boundary, as the corresponding domains of the Dirac-type operators are not preserved by the representations of the algebra. 

In \cite{KMP}, the authors claimed to construct an even spectral triple over the quantum disk. Due to technicalities in the definition of an implementation of a derivation, however, this was not true. We clarify the generalized concept of implementations of unbounded derivations and when they lead to spectral triples in the next section. In Section 3 we establish the notation and review results from \cite{KMR}. The main result, Theorem \ref{compact_inverse}, is proved in Section 4. It states that, for a class of exponential coefficients, the operator $D$ defined in equation \eqref{D_def_ref} as a suitable Hilbert spaces implementation of a covariant derivation $\delta$ of \eqref{delta_cov}, in the quantum annulus algebra $A$, see \eqref{A_def}, has a compact parametrix; in fact the inverse of $D$ is compact.

\section{Implementations of Unbounded Derivations}
Let $A$ be a $C^*$-algebra, $\mathcal{A}\subseteq A$ a dense $*$-subalgebra and $\delta:\mathcal{A}\mapsto A$ a derivation.
Suppose that $H_1$ and $H_2$ are Hilbert spaces carrying representations of $A$ and denoted $\pi_1$ and $\pi_2$ respectively.
The following is a natural concept of an implementation of a derivation in $A$ between two Hilbert spaces, generalizing the usual notion of an implementation of an unbounded derivation.

An implementation of $\delta$ between $H_1$ and $H_2$ consists of the following:
\begin{itemize}
\item A dense subspace dom$(D)\subseteq H_1$
\item An implementing operator $D: \textrm{dom}(D)\to H_2$
\item An intertwinner $i: \textrm{dom}(D)\to H_2$
\end{itemize}
such that
\begin{enumerate}
\item $\forall a\in \mathcal{A}$, $\forall x\in \textrm{dom}(D)$
$$\pi_1(a)x\in \textrm{dom}(D)$$
\item  $\forall a\in \mathcal{A}$, $\forall x\in \textrm{dom}(D)$
$$i(\pi_1(a)x)=\pi_2(a)i(x)$$
\item  $\forall a\in \mathcal{A}$, $\forall x\in \textrm{dom}(D)$
$$D\pi_1(a)x-\pi_2(a)Dx=\pi_2(\delta(a))i(x)$$
\end{enumerate}

A special case of the above definition is when $H_1=H_2=H$, $\pi_1=\pi_2=\pi$ and $i$ is the identity map. Then the second condition above is obviously satisfied, while the third condition can be written as
\begin{equation*}
(D\pi(a)-\pi(a)D)x=\pi(\delta(a))x.
\end{equation*}
This coincides with the usual concept of an unbounded implementation of a derivation as a commutator.

Recall that a closed operator $D$ is called a Fredholm operator if there are bounded operators $Q_1$ and $Q_2$ such that $Q_1D - I$ and $DQ_2 - I$ are compact. The operators $Q_1$ and $Q_2$ are called left and right parametrices respectively. We say that a Fredholm operator $D$ has compact parametrices if at least one (and consequently both) of the parametrices $Q_1$ and $Q_2$ is compact. More on general properties of operators with compact parametrices can be found in the appendix of \cite{KMRSW}. We also say that an implementation $(\textrm{dom}(D), D, i)$ has compact parametrices if the closure of the operator $D$ has compact parametrices.

Under additional conditions an implementation of a derivation can lead to an even spectral triple over $A$. Namely, 
defining $H = H_1 \bigoplus H_2$, with grading $\Gamma \big|_{H_1} =1$ and $\Gamma \big|_{H_2} = -1$ and a representation $\pi :A \to B(H)$ of $A$ in $H$ given by the formula:
$$\pi(a) = (\pi_1(a),\pi_2(a)),$$ 
and also defining a generally unbounded operator $\dcal$ in $H$ by:
\begin{equation*}
\dcal = \left[
\begin{array}{cc}
0 & D \\
D^* & 0
\end{array}\right],
\end{equation*}
we see that $\pi(a)$ are even and $\dcal$ is odd with respect to grading $\Gamma$. If $\dcal$ is a self-adjoint operator with compact parametrix and additionally the intertwinner $i$ is bounded, then the conditions in the definition of a derivation implementation imply that $\pi(a)$ preserve the domain of $\dcal$ for all $a\in \mathcal{A}$ and the commutator $[\dcal,\pi(a)]$ is bounded as can be seen by a simple calculation. Consequently, under those additional conditions, we obtain an even spectral triple over $A$.

A very natural class of implementations of derivations can be obtained from GNS representations in the following way. Suppose $\tau_1$, $\tau_2$ are faithful states on $A$. Let $H_1$, $H_2$ be the corresponding GNS Hilbert spaces, obtained by completing $A$ with respect to the inner products
\begin{equation*}
(a,b)_i=\tau_i(a^*b),\ \ i=0,1.
\end{equation*}
Because we assume that the states are faithful, $A$ sits densely in $H_1$, $H_2$. More precisely, there are injective continuous linear maps $\phi_1:A\to H_1$, $\phi_1:A\to H_1$ with dense ranges embedding $A$ into $H_1$, $H_2$. 

The Hilbert spaces $H_1$, $H_2$ carry natural representations of $A$ given by left multiplication; for $a,b\in A$ we have
\begin{equation*}
\pi_i(a)\phi_i(b)=\phi_i(ab).
\end{equation*}
Suppose as before that we have a dense $*$-subalgebra $\mathcal{A}\subseteq A$ and a derivation $\delta:\mathcal{A}\mapsto A$. Then we have the following implementation of $\delta$ between $H_1$ and $H_2$:
\begin{itemize}
\item dom$(D):=\phi_1(\mathcal{A})\subseteq H_1$. If $x\in  \textrm{dom}(D)$ then we write $x=\phi_1(b)$ for some $b\in \mathcal{A}$, notation we use in the formulas below.
\item $D(x):=\phi_2(\delta(b))$, 
\item $i(x):=\phi_2(b)$
\end{itemize}
It is a matter of straightforward calculations to verify that indeed the three conditions of the definition are satisfied and the above defines an implementation of $\delta$ between $H_1$ and $H_2$.

\section{Quantum Annulus Preliminaries}
We review the notation and basic concepts from \cite{KMR} below and, in a number of places, we use the results contained in that paper.

\subsection{The Quantum Annulus}
Let  $\{E_l\}_{l\in\Z}$ be the canonical basis for $\ell^2(\Z)$ and $V$ be the bilateral shift defined by 
\begin{equation*}
VE_l = E_{l+1}\,.
\end{equation*} 
Notice that $V$ is a unitary.  Let $\mathbb{L}$ be the diagonal label operator defined by
\begin{equation*}
\mathbb{L}E_l = lE_l\,.
\end{equation*}
It follows from the functional calculus that given a function $a:\Z \to \C$, we have
\begin{equation*}
a(\mathbb{L})E_l = a(l)E_l\,.
\end{equation*}
These are precisely the operators which are diagonal with respect to $\{E_l\}$. The operators $(\mathbb{L},V)$ serve as noncommutative polar coordinates, and they satisfy the following commutation relation:
\begin{equation*}
\mathbb{L}V=V(\mathbb{L}+I)\,.
\end{equation*}
Let $c(\Z)$ be the set of $a(l)$, as above, which are convergent as $l\to\pm\infty$ and let $c_{00}^{++}(\Z)$ be the set of all eventually constant functions, i.e. functions $a(l)$ such that there exists a $l_0$ where $a(l)$ is a constant for $l\geq l_0$ and also is a possibly different constant for $l\leq -l_0$. 

Let $A$ be the $C^*$-algebra generated by $V$ and $a(\mathbb{L})$, that is:
 \begin{equation}\label{A_def}
A = C^*(V,a(\mathbb{L}): a(l)\in c(\Z)) 
\end{equation}
This algebra is called the quantum annulus.  The smallest reasonable domain of derivations in $A$ is the following dense $*$-subalgebra of A:
\begin{equation*}
\mathcal{A} = \left\{a = \sum_{n\in\Z} V^n a_n(\mathbb{L}) : \ a_n(l)\in c_{00}^{++}(\Z),\  \textrm{finite sums}\right\}\,.
\end{equation*}

\subsection{Derivations in the Quantum Annulus}
Let $\rho_{\theta}: A \to A $, $0\leq\theta<2\pi$, be a one parameter group of automorphisms of A defined by: 
$$\rho_\theta (a) = e^{i\theta\mathbb{L}}ae^{-i\theta\mathbb{L}}.$$
Since $\rho_\theta (a(\mathbb{L})) = a(\mathbb{L})$, $\rho_\theta (V) = e^{i\theta}V$ and consequently $\rho_\theta(V^{-1}) = e^{-i\theta}V^{-1}$, the automorphisms $\rho_\theta$ are well defined on $A$ and they preserve $\mathcal{A}$. By Proposition 3.2 in \cite{KMR}, any densely-defined derivation $\delta:\mathcal{A} \to A$, covariant with respect to  $\rho_\theta$ that is
\begin{equation*}
\rho_\theta(\delta(a)) = e^{i\theta}\delta(\rho_\theta(a))\,,
\end{equation*}
is of the following form:
\begin{equation}\label{delta_cov}
\delta(a) = [U\beta(\mathbb{L}),a]\,
\end{equation}
where $\{\beta(l+1) - \beta(l)\}\in c(\Z)$. We use notation:
\begin{equation*}
\lim_{l \to \pm\infty}(\beta(l+1) - \beta(l)) := \beta_{\pm\infty},
\end{equation*}
and below we only consider covariant derivations with $\beta_{\pm\infty}\ne 0$. It follows that there are constants $c_1$ and $c_2$ so that
\begin{equation}\label{beta_bound}
c_1(|l|+1)\le |\beta(l)|\le c_2(|l| + 1)\,.
\end{equation}

\subsection{Covariant Implementations on the Quantum Annulus}
Here we consider covariant implementations of derivations \eqref{delta_cov}. We begin by introducing the following family of states $\tau_w:A \to \C$ on $A$, defined by 
\begin{equation*}
\tau_w (a) = \textrm{tr}(w(\mathbb{L})a)\,,
\end{equation*} 
where $w(l) > 0$ for all $l \in \Z$ and 
\begin{equation*}
\sum_{l\in\Z}{w(l)}=1\,.
\end{equation*}

As a result of Proposition 4.3 in \cite{KMR}, $\tau_w$ are precisely the $\rho_\theta$-invariant, normal, faithful states on A. Let $H_w$ be the Hilbert space obtained by Gelfand-Naimark-Segal (GNS) construction on $A$ using state $\tau_w$. Since the state is faithful, $H_w$ is the completion of $A$ with respect to the inner product given by 
\begin{equation*}
\langle a,b\rangle_w = \tau_w(a^*b)\,.
\end{equation*} 
A simple calculation leads to the following precise description: 
\begin{equation}\label{f_expansion}
H_w = \left\{f = \sum_{n\in\Z}V^n f_n(\mathbb{L}) : \|f\|_w^2= \sum_{n \in\Z} \sum_{l\in\Z} w(l)|f_n(l)|^2 <\infty\right\}\,.
\end{equation} 
With this identification we naturally have $A\subseteq H_w$ so the inclusion maps $\phi_w: A\to H_w$ are the identity maps.
Notice also that $\mathcal{A}$ is dense in $H_w$. 
The GNS representation map $\pi_w : A \to B(H_w)$ is given by left-hand multiplication:
\begin{equation*}
\pi_w (a)f = af\,.
\end{equation*}

Define a one parameter group of unitary operators $V_\theta ^w:H_w \to H_w$ via the formula: 
\begin{equation*}
V_\theta^wf = \sum_{n\in\Z}V^ne^{in\theta}f_n(\mathbb{L})\,.
\end{equation*}
An immediate calculation shows that:
\begin{equation*}
\pi_w(\rho_\theta(a)) = V_\theta^w\pi_w(a)(V_\theta^w)^{-1}\,,
\end{equation*}
and therefore the operators $V_\theta^w$ are implementing the one parameter group of automorphisms $\rho_\theta$.

Consider an additional weight, $w'(l)$, possibly different from $w(l)$, satisfying the same conditions. Proceeding like in the previous section, we set
\begin{equation*}
\textrm{dom}(D) := \mathcal A \subset H_{w}\,,
\end{equation*}
and choose for an implementing operator
\begin{equation*}
i:\textrm{dom}(D) = \mathcal A\to \mathcal A \subset H_{w'}\,
\end{equation*}
to be the identity operator $a\mapsto a$. Clearly, the first two properties of an implementation are satisfied. We say that
an operator $D:H_w\supseteq \mathcal{A} \to H_{w'}$ defines a {\it covariant implementation} of the covariant derivation \eqref{delta_cov} if for every $a\in\mathcal{A}$, and for every $f\in\mathcal{A}$ considered as an element of both $H_w$ and $H_{w'}$, we have:
\begin{equation*}
D\pi_w(a)f -\pi_{w'}(a)Df = \pi_{w'}(\delta(a))f\,,
\end{equation*}
and, additionally, $D$ satisfies:
\begin{equation*}
V_\theta^{w'} D (V_\theta^w)^{-1}f=e^{i\theta}Df\,.
\end{equation*}

Proceeding as in Proposition 5.2 in \cite{KMR} shows the following result.
\begin{prop}\label{D_formula}
There exists a sequence $\{\alpha(l)\}$ satisfying
\begin{equation}\label{alpha_condition}
\sum_{l\in\Z}|\beta(l)-\alpha(l)|^2w'(l)<\infty
\end{equation}
such that any covariant implementation $D:H_w\supseteq \mathcal{A} \to H_{w'}$ is of the form:
\begin{equation}\label{D_def_ref}
Df = V\beta(\mathbb{L})f - fV\alpha(\mathbb{L})\,.
\end{equation}
Conversely, for any $\{\alpha(l)\}$ satisfying \eqref{alpha_condition}, the formula \eqref{D_def_ref} defines a covariant implementation $D:H_w\supseteq \mathcal{A} \to H_{w'}$ of the derivation \eqref{delta_cov}.
\end{prop}

We assume below that for every $l$ we have:
\begin{equation*}
\alpha(l), \beta(l)\neq 0\,.
\end{equation*}
It is convenient, like in \cite{KMRSW} and \cite{KMR}, to write
\begin{equation*}
\alpha(l) = \beta(l)\frac{\mu(l+1)}{\mu(l)}
\end{equation*}
for some sequence $\{\mu(l)\}$ such that $\mu(0)=1$.

\subsection{Fourier Decomposition}
To further analyze the operator $D$ of formula \eqref{D_def_ref}, we can decompose it into a Fourier series and study its Fourier components which are operators acting between weighted $\ell^2$-spaces, defined as follows:
\begin{equation*}
\ell^2_w=\left\{\{f(l)\}_{l\in\Z}: \sum_{l\in\Z}|f(l)|^2w(l)<\infty\right\}\,.
\end{equation*}
We have the following decomposition proposition:

\begin{prop}\label{D_decomp}
Let $f\in\textrm{dom}(D)$.  Then
\begin{equation*}
Df = \sum_{n\in\Z}V^{n+1}(D_nf_n)(\mathbb{L})\,, 
\end{equation*}
where $D_n:\ell_w^2\supseteq c_{00}^{++}(\Z) \to \ell_{w'}^2$ and is given by the following formula:
\begin{equation*}
(D_nh)(l) = \beta(l+n)h(l) - \beta(l)\frac{\mu(l+1)}{\mu(l)}h(l+1)
\end{equation*}
for some $h\in c_{00}^{++}(\Z)$.
\end{prop}

\begin{proof}
The proof follows by writing $f\in\textrm{dom}(D)$ as its Fourier series, applying $D$ to it and using the commutation relation $\mathbb{L}V = V(\mathbb{L}+I)$.
\end{proof}
In what follows, the purpose is to choose the parameters such that $D$  has a compact parametrix.

\subsection{Parametrices}

Next we study a formal candidate for a parametrix of $D$.  The closure of $D$, defined as above on $c_{00}^{++}(\Z)$, will be denoted by $\bar D$ while its closure defined on the space $c_{00}(\Z)$ of eventually zero functions will be denoted by $\bar D_{00}$. Also notice that $D$ preserves $c_{00}(\Z)$.
We have the following simple observation.
\begin{prop}  If $\{\alpha(l)\}$ satisfies \eqref{alpha_condition} then $\bar D=\bar D_{00}$.
\end{prop}
\begin{proof}Notice that $c_{00}(\Z)\subset \textrm{dom}(D)$ is a co-dimension 2 subspace. Thus, it is enough to verify that 1 and the characteristic function of $\Z_{\ge 0}$ are in the domain of $\bar D$. Approximating 1 by characteristic functions $\chi_N$ of sets $-N\le l\le N$, $\chi_N\in c_{00}(\Z)$, we see that $D(\chi_N)$ converges in $\ell_w^2$ to $D(1)$, which is in $\ell_w^2$ by \eqref{alpha_condition}, implying that 1 is in the closure of $D$. The characteristic function of $\Z_{\ge 0}$ is in the domain of $\bar D$ by the same argument.
\end{proof}

%
%
%
%

Let $Q_n$ be given by the following formula:
\begin{equation*}
(Q_ng)(l) = \left\{
\begin{aligned}
&\sum_{j=l}^\infty\frac{\prod_{k=l}^{l+n-1}\beta(k)}{\prod_{k=j}^{j+n}\beta(k)}\cdot\frac{\mu(j)}{\mu(l)}g(j) &&n\ge0 \\
&\sum_{j=l}^\infty\frac{\prod_{k=j+n+1}^{j-1}\beta(k)}{\prod_{k=l+n}^{l-1}\beta(k)}\cdot\frac{\mu(j)}{\mu(l)}g(j)  &&n<0.
\end{aligned}\right.
\end{equation*}
This expression was obtained by inverting $D_n$ using techniques similar to the calculations in Proposition 4.12 in \cite{KM}. Notice that we have: 
$$Q_n: c_{00}(\Z)\to \textrm{dom}(D),$$ 
since the sums in the definitions of $Q_n$ are finite for $g\in  c_{00}(\Z)$ and the outcomes are eventually constant. Relations between $D_n$ and $Q_n$ are explained in the following statements.

\begin{prop} For every $f\in c_{00}(\Z)$ we have:
\begin{equation*}
D_nQ_nf=f\ \textrm{ and } \ Q_nD_nf=f\,.
\end{equation*}

\end{prop}

\begin{proof} The formulas follow from straightforward calculations.
\end{proof}

From this proposition we can formally define the inverse for $D$ to be $Q = \sum_{n\in\Z}Q_n$.  For this to be well-defined, the series needs to converge.  In fact, once we verify that $Q$ is bounded, the previous two propositions imply that $Q$ is the inverse of $\bar D$.

\section{Results}

 For the remainder of this section we assume that $\beta(k)= k + \frac{1}{2}$.  Moreover we only consider the special choices of the weights $\{w(l)\}$, $\{w'(l)\}$ and the choice for $\{\mu(l)\}$ namely:
\begin{equation*}
w(l) = e^{-a|l|},\quad w'(l) = e^{-b|l|},\quad\textrm{and }\mu(l) = e^{-(\gamma l)/2},\textrm{ for }a,b, \gamma>0\,.
\end{equation*}
It should be noted that with these specific choices $\alpha(l)= e^{-\gamma /2}\beta(l)$ and the conditions in Proposition \ref{D_formula} are trivially satisfied. By a simple perturbative argument the results below are valid for a much larger class of coefficients, see the remark at the end of the next subsection.

\subsection{Compactness of Parametrices}
We show that $Q_n$ are Hilbert-Schmidt operators for every $n$ and verify that their respective Hilbert-Schmidt norms go to zero as $n$ goes to infinity, implying that $Q$ is a compact operator.  To prove this we need a few helper lemmas. We postpone proofs of those lemmas until the next subsection.

\begin{lem}\label{lem1}
For $0\le l\le n$, define the following product:
\begin{equation*}
q_n(l) = \frac{\left(\frac{1}{2}\right)^2\left(\frac{3}{2}\right)^2\cdots\left(n-\frac{1}{2}\right)^2}{\left(l-\frac{1}{2}\right)^2\left(l-\frac{3}{2}\right)^2\cdots\left(l-n+\frac{1}{2}\right)^2}\,.
\end{equation*}
Then we have the identity:
\begin{equation*}
q_n(l) = \frac{((2n-2l+1)\cdots(2n-3)(2n-1))^2}{(1\cdot 3\cdot5\cdots (2l-1))^2}\textrm{ for every natural number }n\,.
\end{equation*}
Moreover, $q_n(l)$ satisfies the following three estimates:
\begin{equation*}
(1)\ q_n(l)\ge1,
\quad (2)\ q_n(l)\le 2l\left(
\begin{array}{c}
2n \\ 2l
\end{array}\right),
\quad (3)\ \frac{q_n(j)}{q_n(l)}\le \left(
\begin{array}{c}
2n \\ 2l
\end{array}\right)\textrm{ for }0\le j\le l\,.
\end{equation*}
\end{lem}

\begin{lem}\label{lem3}
For a nonnegative integer $j$, define the following sum:
\begin{equation*}
J_n(j) = \sum_{k\ge0}\frac{\left(k+j+\frac{1}{2}\right)^2\cdots\left(k+j+n-\frac{1}{2}\right)^2e^{-\gamma k}}{\left(j+\frac{1}{2}\right)^2\cdots\left(j+n-\frac{1}{2}\right)^2}\,.
\end{equation*}
Then for $n\ge0$
\begin{equation*}
J_n(j)\le \frac{2n+1}{\left(1-e^{-\frac{\gamma}{2}}\right)^{2n+1}}\,.
\end{equation*}
\end{lem}

The following is the main technical result of the paper.

\begin{theo}\label{compact_inverse}
Suppose that $\gamma>a>b$ and that
\begin{equation*}
\textrm{exp}\left(-\frac{(a-b)}{2}\right) + \textrm{exp}\left(-\frac{\gamma + a}{2}\right)<1\,.
\end{equation*}
Then $Q_n:\ell_{w'}^2\to\ell_w^2$ is a Hilbert-Schmidt operator for every $n\in\Z$ and $\|Q_n\|_{\textrm{HS}}\to0$ as $n\to\pm\infty$.  Consequently, $Q=\sum_n Q_n$ is the inverse of $\bar D$ and is a compact operator.
\end{theo}

\begin{proof}
Notice that formula for $Q_n$ shows that it is an integral operator.  Therefore, by direct calculation we can compute the Hilbert-Schmidt norms:
\begin{equation*}
\|Q_n\|_{\textrm{HS}}^2=\left\{
\begin{aligned}
&\sum_{j=l}^\infty\frac{\prod_{k=l}^{l+n-1}|\beta(k)|^2}{\prod_{k=j}^{j+n}|\beta(k)|^2}\cdot\frac{|\mu(j)|^2}{|\mu(l)|^2}\cdot\frac{w(l)}{w'(j)} &&n\ge0
 \\
&\sum_{j=l}^\infty\frac{\prod_{k=j+n+1}^{j-1}|\beta(k)|^2}{\prod_{k=l+n}^{l-1}|\beta(k)|^2}\cdot\frac{|\mu(j)|^2}{|\mu(l)|^2}\cdot\frac{w(l)}{w'(j)}  &&n<0
\end{aligned}\right.
\end{equation*}

 Since the choice of $\beta$'s are linear, the following ratios:
\begin{equation*}
\left|\frac{\beta(l)\cdots\beta(l+n-1)}{\beta(l+n)\cdots\beta(l-1)}\right|^2\textrm{ and }\left|\frac{\beta(j+n+1)\cdots\beta(j-1)}{\beta(l+n)\cdots\beta(l-1)}\right|^2
\end{equation*}
are ratios of polynomials.  Thus by our choice of exponential $\mu$'s, $w$'s and $w'$'s, $\|Q_n\|_{\textrm{HS}}$ exists if and only if
\begin{equation*}
\sum_{j\ge l}\left|\frac{\mu(j)}{\mu(l)}\right|^2\cdot\frac{w(l)}{w'(j)}<\infty\,.
\end{equation*}
However, this easily follows if $\gamma>a>b$.  Thus the Hilbert-Schmidt norm of $Q_n$ exists for all $n\in\Z$.  It remains to show that those norms go to zero as $n\to\pm\infty$.  We only need to study the case for $n\ge0$ as if $n<0$ then by doing a change of variables of $j\mapsto -l$, $l\mapsto -j$ and $n\mapsto -n-1$ we are back in the $n\ge0$ case. 
Thus, we only need to estimate the sum:
\begin{equation*}
\|Q_n\|_{\textrm{HS}}^2 = \sum_{j\ge l}\frac{\left(l+\frac{1}{2}\right)^2\cdots\left(l+n-\frac{1}{2}\right)^2}{\left(j+\frac{1}{2}\right)^2\cdots\left(j+n-\frac{1}{2}\right)^2}\cdot\frac{e^{-\gamma(j-l)-a|l|+b|j|}}{\left(j+n+\frac{1}{2}\right)^2}\,.
\end{equation*}
The sum is over all $j\ge l$.  It splits into sums over four main regions which we will further subdivide as illustrated in the picture below:

\begin{figure}[ht]
\begin{center}
\includegraphics[width=65truemm]{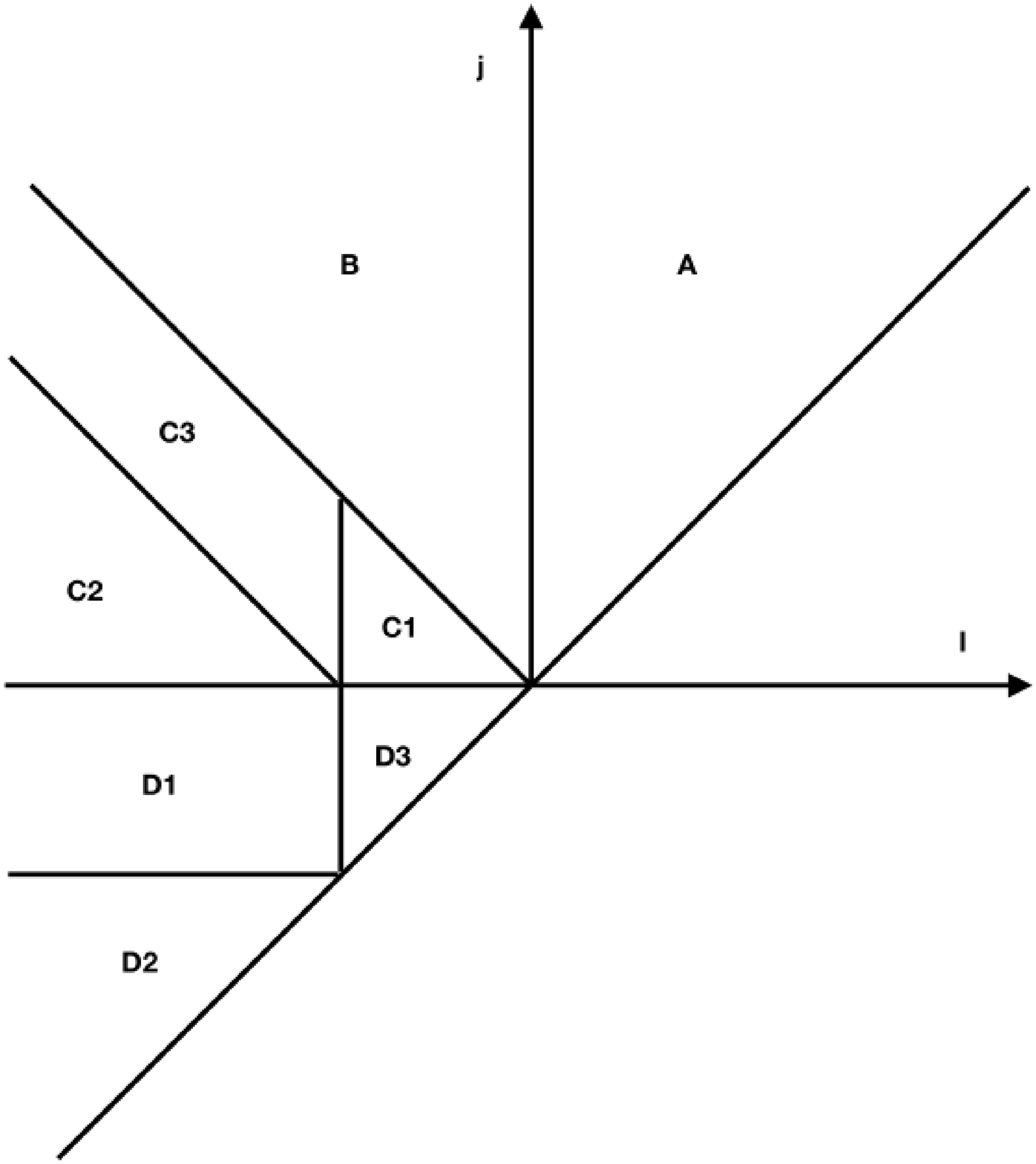}
\end{center}
\end{figure}

We have
\begin{equation*}
\|Q_n\|_{\textrm{HS}}^2 \leq S_A+S_B+S_C+S_D,
\end{equation*}
as for simplicity of estimations we let the regions overlap.
Here the regions are: region A: $j\ge l \ge0$, region B: $j\ge-l\geq0$, $l\le0$, region C: $-l\ge j\ge0$, $l\le0$, and region D: $l\ge j\ge0$, which we will handle separately.  In our estimates, we double count the boundaries in some places for convenience of different estimates.

{\bf Regions A and B:}   Notice the following observation: if $-j\le l\le j$ for $j\ge0$, then for any $r>0$ we have that $(l+r)^2\le (j+r)^2$.  Using this fact, 
we can estimate:
\begin{equation*}
\frac{\left(l+\frac{1}{2}\right)^2\cdots\left(l+n-\frac{1}{2}\right)^2}{\left(j+\frac{1}{2}\right)^2\cdots\left(j+n-\frac{1}{2}\right)^2}\le1.
\end{equation*}
It follows that we have
\begin{equation*}
S_A\le \sum_{j\ge l\ge0}\frac{e^{-\gamma(j-l)-al+bj}}{\left(j+n+\frac{1}{2}\right)^2}\to0\textrm{ as }n\to\infty\,.
\end{equation*}
Changing $l\to -l$ we get exactly the same estimate for $S_B$, and so $S_B\to0$ as $n\to\infty$.

{\bf Region C:} $0\le j\le -l$.  

First we map $l\mapsto -l$, then we split this region into two sub-regions: C1: $0\le j\le l\le n$ and the complement of C1: $0\le j\le l$, $l\ge n$.  In the second region we make a change of variables of $l'=l-n$, then replace $l'$ with $l$ and we further get two additional sub-regions: C2: $0\le j\le l$ and C3: $0\le l\le j\le l+n$. 

First in sub-region C1: since $l\le n$, using the first inequality from Lemma \ref{lem1}, we have
\begin{equation*}
\begin{aligned}
S_{C1} &=\sum_{0\le j\le l\le n}\frac{\left(l-\frac{1}{2}\right)^2\cdots\left(l-n+\frac{1}{2}\right)^2}{\left(j+\frac{1}{2}\right)^2\cdots\left(j+n-\frac{1}{2}\right)^2}\cdot\frac{e^{-\gamma(j+l)-al+bj}}{\left(j+n+\frac{1}{2}\right)^2} \\
&\le\sum_{0\le j\le l\le n}\frac{\left(l-\frac{1}{2}\right)^2\cdots\left(l-n+\frac{1}{2}\right)^2}{\left(\frac{1}{2}\right)^2\cdots\left(n-\frac{1}{2}\right)^2}\cdot\frac{e^{-\gamma(j+l)-al+bj}}{\left(j+n+\frac{1}{2}\right)^2} \\
&\le\sum_{0\le j\le l\le n}\frac{e^{-\gamma(j+l)-al+bj}}{\left(j+n+\frac{1}{2}\right)^2}\le\frac{1}{\left(n+\frac{1}{2}\right)^2}\sum_{0\le j\le l<\infty}e^{(-\gamma +b)j-(\gamma+a)l} <\frac{\textrm{const}}{\left(n+\frac{1}{2}\right)^2}
\end{aligned}
\end{equation*}
which clearly goes to zero as $n\to\infty$.

In the case of C2, we get
\begin{equation*}
S_{C2} = \sum_{0\le j\le l}\frac{\left(l+\frac{1}{2}\right)^2\cdots\left(l+n-\frac{1}{2}\right)^2}{\left(j+\frac{1}{2}\right)^2\cdots\left(j+n-\frac{1}{2}\right)^2}\cdot\frac{e^{-(\gamma +a)n+(-\gamma +b)j - (\gamma +a)l}}{\left(j+n+\frac{1}{2}\right)^2}\,.
\end{equation*}
Letting $l=k+j$, the sum becomes
\begin{equation*}
S_{C2} = \sum_{j,k\ge0}\frac{\left(k+j+\frac{1}{2}\right)^2\cdots\left(k+j+n-\frac{1}{2}\right)^2}{\left(j+\frac{1}{2}\right)^2\cdots\left(j+n-\frac{1}{2}\right)^2}\cdot\frac{e^{-(\gamma +a)n+(-\gamma +b)j - (\gamma +a)(k+j)}}{\left(j+n+\frac{1}{2}\right)^2}\,.
\end{equation*}
Implementing Lemma \ref{lem3} we have
\begin{equation*}
S_{C2} \le \frac{(2n+1)e^{-(\gamma+a)n}}{\left(1 - e^{-\frac{\gamma +a}{2}}\right)^{2n+1}}\to0 \textrm{ as }n\to\infty\,.
\end{equation*}

In sub-region C3 we have:
\begin{equation*}
S_{C3} = \sum_{0\le l\le j\le l+n} \frac{\left(l+\frac{1}{2}\right)^2\cdots\left(l+n-\frac{1}{2}\right)^2}{\left(j+\frac{1}{2}\right)^2\cdots\left(j+n-\frac{1}{2}\right)^2}\cdot\frac{e^{-(\gamma-b)j-(\gamma+a)(l+n)}}{\left(j+n+\frac{1}{2}\right)^2}
\end{equation*}
Notice that in this region we again have that
\begin{equation*}
\frac{\left(l+\frac{1}{2}\right)^2\cdots\left(l+n-\frac{1}{2}\right)^2}{\left(j+\frac{1}{2}\right)^2\cdots\left(j+n-\frac{1}{2}\right)^2}\le1.
\end{equation*}
Thus, by overestimating, we have that
\begin{equation*}
S_{C3}\le \sum_{0\le l\le j\le\infty} \frac{e^{-(\gamma-b)j-(\gamma+a)(l+n)}}{\left(j+n+\frac{1}{2}\right)^2} \le \frac{1}{\left(n+\frac{1}{2}\right)^2}\sum_{0\le l\le j\le\infty} e^{-(\gamma-b)j-(\gamma+a)(l+n)}
\end{equation*}
which goes to zero as $n\to\infty$.

{\bf Region D:} $l\le j\le0$.  

First we map $j\mapsto -j$ and $l\mapsto -l$ and the sum becomes
\begin{equation*}
S_{D} = \sum_{0\le j\le l}\frac{\left(l-\frac{1}{2}\right)^2\cdots\left(l-n+\frac{1}{2}\right)^2}{\left(j-\frac{1}{2}\right)^2\cdots\left(j-n+\frac{1}{2}\right)^2}\cdot\frac{e^{\gamma j - \gamma l -a|l| + b|j|}}{\left(j-n-\frac{1}{2}\right)^2}\,.
\end{equation*}
Like with region C, we split this into three sub-regions: D1: $0\le j\le n\le l$, D2: $n\le j\le l$, and D3: $0\le j\le l\le n$.

In sub-region D1, we map $l\mapsto l+n$ which yields
\begin{equation*}
S_{D1} = \sum_{0\le j\le n,l\ge0}\frac{\left(l+\frac{1}{2}\right)^2\cdots\left(l+n-\frac{1}{2}\right)^2}{\left(j-\frac{1}{2}\right)^2\cdots\left(j-n+\frac{1}{2}\right)^2}\cdot\frac{e^{-(\gamma + a)n}e^{(\gamma +b)j - (\gamma +a)l}}{\left(j-n-\frac{1}{2}\right)^2}\,.
\end{equation*}
Multiplying and dividing the sum by $(\frac{1}{2})^2\cdots(n-\frac{1}{2})^2$ and using Lemma \ref{lem3} we get
\begin{equation*}
S_{D1} \le \frac{(2n+1)e^{-(\gamma+a)n}}{\left(1 - e^{-\frac{\gamma+a}{2}}\right)^{2n+1}}\sum_{0\le j\le n}\frac{\left(\frac{1}{2}\right)^2\cdots\left(n-\frac{1}{2}\right)^2}{\left(j-\frac{1}{2}\right)^2\cdots\left(j-n+\frac{1}{2}\right)^2}\cdot\frac{e^{(\gamma+b)j}}{\left(j-n-\frac{1}{2}\right)^2}\,.
\end{equation*}
The second inequality in Lemma \ref{lem1} implies that
\begin{equation*}
\begin{aligned}
S_{D1} &\le \frac{(2n+1)e^{-(\gamma+a)n}}{\left(1 - e^{-\frac{\gamma+a}{2}}\right)^{2n+1}}\sum_{0\le j\le n}\left(
\begin{array}{c}
2n \\2j
\end{array}\right)\frac{2je^{(\gamma+b)j}}{\left(j-n-\frac{1}{2}\right)^2} \\
&\le \frac{(2n+1)e^{-(\gamma+a)n}}{\left(1 - e^{-\frac{\gamma+a}{2}}\right)^{2n+1}}\sum_{0\le j\le n}\left(
\begin{array}{c}
2n \\j
\end{array}\right)e^{\left(\frac{\gamma+b}{2}\right)j} \\
&\le n(2n+1)\left(\frac{e^{-\frac{\gamma+a}{2}}+e^{-\frac{(a-b)}{2}}}{1-e^{-\frac{\gamma+a}{2}}}\right)^{2n}\,.
\end{aligned}
\end{equation*}
The conditions on $a$, $b$ and $\gamma$ imply that the right hand side of the above inequality goes to $0$ as $n\to\infty$ and thus $S_{D1}\to0$ as $n\to\infty$.

In sub-region D2, we map $l\mapsto l+n$ and $j\mapsto j+n$ to get
\begin{equation*}
S_{D2}=\sum_{0\le j\le l}\frac{\left(l+\frac{1}{2}\right)^2\cdots\left(l+n-\frac{1}{2}\right)^2}{\left(j+\frac{1}{2}\right)^2\cdots\left(j+n-\frac{1}{2}\right)^2}\cdot\frac{e^{-(\gamma + a)(l+n) +(\gamma+b)(j+n)}}{\left(j-\frac{1}{2}\right)^2}\,.
\end{equation*}
Writing $l=k+j$ and using Lemma \ref{lem3} we obtain:
\begin{equation*}
\begin{aligned}
S_{D2}&=e^{-(a-b)n}\sum_{0\le j,k}\frac{\left(j+k+\frac{1}{2}\right)^2\cdots\left(j+k+n-\frac{1}{2}\right)^2}{\left(j+\frac{1}{2}\right)^2\cdots\left(j+n-\frac{1}{2}\right)^2}\cdot\frac{e^{-(a-b)j -(\gamma+a)k}}{\left(j-\frac{1}{2}\right)^2}\\
&\le \frac{(2n+1)e^{-(a-b)n}}{\left(1 - e^{-\frac{\gamma+a}{2}}\right)^{2n+1}}\sum_{0\le j}\frac{e^{-(a-b)j}}{\left(j-\frac{1}{2}\right)^2}<\infty\,.
\end{aligned}
\end{equation*}
Like in sub-region D1, the conditions on $a$, $b$ and $\gamma$ imply the last term in the above inequality goes to zero as $n\to\infty$ and hence $S_{D2}$ goes to zero as $n\to\infty$.

Finally in the sub-region D3, by multiplying and dividing by $(\frac{1}{2})^2\cdots(n-\frac{1}{2})^2$, we have that:
\begin{equation*}
S_{D3}=\sum_{0\le j\le l\le n}\frac{q_n(j)}{q_n(l)}\cdot\frac{e^{-(\gamma+a)l+(\gamma+b)j}}{\left(j-n-\frac{1}{2}\right)^2}\,.
\end{equation*}

Observe in this region we have that $n+\frac{1}{2}-j>n+\frac{1}{3}-j$.  Using this observation and the third inequality in Lemma \ref{lem1} we have
\begin{equation*}
\begin{aligned}
S_{D3} &\le\sum_{0\le j\le l}\left(
\begin{array}{c}
2l\\2j
\end{array}\right)\cdot\frac{e^{-(\frac{\gamma+a}{2})2l+(\frac{\gamma+b}{2})2j}}{\left(\frac{1}{2}\cdot 2j-n-\frac{1}{3}\right)^2}\le \sum_{0\le j\le l\le 2n}\left(
\begin{array}{c}
l\\j
\end{array}\right)\cdot\frac{e^{-(\frac{\gamma+a}{2})l+(\frac{\gamma+b}{2})j}}{\left(\frac{j}{2}-n-\frac{1}{3}\right)^2}\\
&=\sum_{l\ge0}\left(\sum_{j=0}^l\left(
\begin{array}{c}
l\\j
\end{array}\right)\frac{e^{(\frac{\gamma+b}{2})j}}{\left(\frac{j}{2}-n-\frac{1}{3}\right)^2}\right)e^{-(\frac{\gamma+a}{2})l}
=\sum_{l\ge0}\frac{1}{\left(\frac{l}{2}-n-\frac{1}{3}\right)^2}\left(e^{-(\frac{\gamma+a}{2})}+e^{-(\frac{a-b}{2})}\right)^l
\end{aligned}
\end{equation*}
where the last sum is finite.  Thus by the Lebesgue Dominated Convergence Theorem, $S_{D2}\to0$ as $n\to\infty$. 
This completes the proof.
\end{proof}

\noindent {\bf Remark.}
Since a bounded perturbation of an operator with compact parametrices also has compact parametrices, see Appendix of \cite{KMRSW}, it follows for example that,
if $\beta(l)=\beta_\infty l+\tilde\beta(l)$ and $\alpha(l)= e^{-\gamma /2}\beta_\infty l+\tilde\alpha(l)$, where $\beta_\infty$ is a nonzero constant, $\tilde\alpha(l)$ and $\tilde\beta(l)$ are bounded, then the corresponding operator $D$ has compact parametrices. This observation substantially increases the class of covariant implementations with compact parametrices.

\subsection{Proofs of Lemmas}

\begin{proof} (of Lemma \ref{lem1})
Notice that
\begin{equation*}
q_n(1) = \frac{\left(\frac{1}{2}\right)^2\left(\frac{3}{2}\right)^2\cdots\left(n-\frac{3}{2}\right)^2\left(n-\frac{1}{2}\right)^2}{\left(\frac{1}{2}\right)^2\left(\frac{1}{2}\right)^2\left(\frac{3}{2}\right)^2\cdots\left(n-\frac{5}{2}\right)^2\left(n-\frac{3}{2}\right)^2} = \frac{(2n-1)^2}{1^2}\,.
\end{equation*}
Similarly
\begin{equation*}
q_n(2) = \frac{(2n-3)^2(2n-1)^2}{1^2\cdot 3^2}\,.
\end{equation*}
It follows by induction that
\begin{equation}\label{ind_formula}
q_n(l) = q_n(l-1)\frac{(2n-2l+1)^2}{(2l-1)^2} =\frac{((2n-2l+1)\cdots(2n-3)(2n-1))^2}{(1\cdot 3\cdot5\cdots (2l-1))^2}\,.
\end{equation}
For the first inequality, notice that from the inductive formula \eqref{ind_formula} we see that $q_n(l)$ as a function of $l$, $0\le l\le n$, first increases and then decreases, so the minimum of it occurs at the endpoints.  This means that $q_n(l)\ge q_n(0)=q_n(n)=1$, yielding the first inequality.  

To prove the second inequality notice that
\begin{equation*}
\begin{aligned}
q_n(l) &= \frac{((2n-2l+1)\cdots(2n-3)(2n-1))^2}{(1\cdot 3\cdot5\cdots (2l-1))^2} \\
&\le \frac{(2n-2l+1)(2n-2l+2)\cdots(2n-1)(2n)}{1\cdot2\cdot3\cdots(2l-2)(2l-1)} = \frac{(2n)!}{(2n-1)!(2n-2l)!}\\
&=\frac{2l(2n)!}{(2l)!(2n-2l)!} = 2l\left(
\begin{array}{c}
2n\\2l
\end{array}\right)\,.
\end{aligned}
\end{equation*}
To prove the final inequality we estimate as follows:
\begin{equation*}
\begin{aligned}
\frac{q_n(j)}{q_n(l)}&=\frac{\left(l-\frac{1}{2}\right)^2\cdots\left(l-n+\frac{1}{2}\right)^2}{\left(j-\frac{1}{2}\right)^2\cdots\left(j-n+\frac{1}{2}\right)^2} = \frac{(2j+1)^2\cdots(2l-1)^2}{(2n-2l+1)^2\cdots(2n-2j+1)^2}\\
&\le\frac{(2j+1)(2j+2)\cdots(2l-1)(2l)}{(2n-2l)(2n-2l+1)\cdots(2n-2j+1)} = \frac{(2l)!}{(2j)!}\cdot\frac{(2n-2l-1)!}{(2n-2j-1)!}\\
&=\frac{\left(
\begin{array}{c}
2l\\2j
\end{array}\right)}{\left(
\begin{array}{c}
2n-2j-1\\2n-2l-1
\end{array}\right)}\le \left(
\begin{array}{c}
2l\\2j
\end{array}\right)\,.
\end{aligned}
\end{equation*}
\end{proof}

To prove Lemma \ref{lem3}, we need the following additional step.

\begin{lem}\label{lem2}
For a nonnegative integer $j$, define the following sum:
\begin{equation*}
I_m(j) = \sum_{k\ge0}\frac{(k+2j+1)\cdots(k+2j+m)e^{-\frac{\gamma}{2}k}}{m!}\textrm{ for }m>0\textrm{ and }I_0 = \left(1 - e^{-\frac{\gamma}{2}}\right)^{-1}\,.
\end{equation*}
Then, for $m\ge 0$, we have:
\begin{equation*}
I_m(j)\le \frac{1}{\left(1 - e^{-\frac{\gamma}{2}}\right)^{m+1}}\cdot\frac{(2j+1)\cdots(2j+m)}{m!}\,.
\end{equation*}
\end{lem}

\begin{proof}
Notice $I_m(j)$ satisfies the following reduction formula:
\begin{equation*}
\begin{aligned}
I_m(j) &= \frac{(2j+1)\cdots(2j+m)}{m!} + e^{-\frac{\gamma}{2}}I_m(j) + \sum_{k\ge1}\frac{(k+2j+1)\cdots(k+2j+m-1)e^{-\frac{\gamma}{2}k}}{(m-1)!} \\
&=\frac{(2j)(2j+1)\cdots(2j+m-1)}{m!} + e^{-\frac{\gamma}{2}}I_m(j) + I_{m-1}(j)\,.
\end{aligned}
\end{equation*}
This implies that $I_m(j)$ satisfies the following recurrence relation:
\begin{equation*}
\left(1-e^{-\frac{\gamma}{2}}\right)I_m(j) - I_{m-1}(j) = \left(
\begin{array}{c}
2j+m-1\\ 2j-1
\end{array}\right)\,.
\end{equation*}
This can be solved recursively which yields
\begin{equation*}
I_m(j) = \frac{1}{\left(1-e^{-\frac{\gamma}{2}}\right)^{m+1}}\sum_{r=0}^m\left(
\begin{array}{c}
2j+r-1\\ r
\end{array}\right)\left(1-e^{-\frac{\gamma}{2}}\right)^r\,.
\end{equation*}
Since $1-e^{-\gamma/2}\le1$ we get
\begin{equation*}
I_m(j) \le \frac{1}{\left(1-e^{-\frac{\gamma}{2}}\right)^{m+1}}\sum_{r=0}^m\left(
\begin{array}{c}
2j+r-1\\ r
\end{array}\right)\,.
\end{equation*}
By parallel summation we have
\begin{equation*}
\sum_{r=0}^m\left(
\begin{array}{c}
2j+r-1\\ r
\end{array}\right) = \left(
\begin{array}{c}
2j+m\\m
\end{array}\right) = \frac{(2j+1)\cdots(2j+m)}{m!}
\end{equation*}
and thus the result follows.
\end{proof}

\begin{proof} (of Lemma \ref{lem3})
Observe the following inequality
\begin{equation*}
\begin{aligned}
J_n(j)&=\sum_{k\ge0}\frac{(2k+2j+1)^2\cdots(2k+2j+2n-1)^2}{(2j+1)^2\cdots(2j+2n-1)^2}e^{-\frac{\gamma}{2}\cdot 2k}\\
&\le\sum_{k\ge0}\frac{(2k+2j+1)(2k+2j+2)\cdots(2k+2j+2n)}{(2j+1)^2\cdots(2j+2n-1)^2}e^{-\frac{\gamma}{2}\cdot 2k}. \\
\end{aligned}
\end{equation*}
Overestimating by adding odd $2k+1$ terms we get:
\begin{equation*}
\begin{aligned}
J_n(j)&\le\frac{(2n)!}{(2j+1)^2\cdots(2j+2n-1)^2}\sum_{k\ge0}\frac{(k+2j+1)\cdots(k+2j+2n)}{(2n)!}e^{-\frac{\gamma}{2}k}\\
&=\frac{(2n)!}{(2j+1)^2\cdots(2j+2n-1)^2}I_{2n}(j),\\
\end{aligned}
\end{equation*}
where $I_{2n}(j)$ is defined in Lemma \ref{lem2}.  By implementing Lemma \ref{lem2} we arrive at
\begin{equation*}
\begin{aligned}
J_n(j)&\le \frac{(2n)!}{(2j+1)^2\cdots(2j+2n-1)^2}\cdot\frac{(2j+1)\cdots(2j+2n)}{(2n)!}\cdot\frac{1}{\left(1-e^{-\frac{\gamma}{2}}\right)^{2n+1}}\\
&=\frac{(2j+2)(2j+4)\cdots(2j+2n)}{(2j+1)(2j+3)\cdots(2j+2n-1)}\cdot\frac{1}{\left(1-e^{-\frac{\gamma}{2}}\right)^{2n+1}}\\
&\le \frac{2j+2n}{2j+1}\cdot\frac{1}{\left(1-e^{-\frac{\gamma}{2}}\right)^{2n+1}}\le\frac{2n+1}{\left(1-e^{-\frac{\gamma}{2}}\right)^{2n+1}}\,.
\end{aligned}
\end{equation*}
\end{proof}

\end{document}